\documentclass[a4paper,12pt]{article}

\usepackage{latexsym,amsmath,enumerate,amssymb,amsthm} 
\usepackage{enumerate,verbatim,tikz} 

\theoremstyle{plain}
\newtheorem{theorem}{Theorem}[section]
\newtheorem{lemma}{Lemma}[section]
\newtheorem{proposition}{Proposition}[section]
\newtheorem{corollary}{Corollary}[section]

\numberwithin{equation}{section}

\theoremstyle{definition}

\theoremstyle{remark}
\newtheorem{remark}{Remark}[section]

\numberwithin{equation}{section}

\newcommand{\I}{\textnormal{I}}
\newcommand{\II}{\textnormal{I\hspace*{-0.4ex}I}}

\newcommand{\Ip}{{\textnormal{I\hspace*{-0.3ex}}^\prime}} 
\newcommand{\IIp}{{\textnormal{I\hspace*{-0.4ex}I\hspace*{-0.3ex}}^\prime}}

\begin{document}

\markboth{Somphong Jitman}
{Good Integers and Applications in Coding Theory}

%
 
%

\title{Good Integers and some Applications in Coding Theory}

\author{Somphong Jitman\\
    \small Department of Mathematics, Faculty of Science, 
    Silpakorn University,\\ \small   Nakhon Pathom 73000,  Thailand.\\ \small {Email: sjitman@gmail.com}}
%
%

\maketitle

\begin{abstract}
	A class of good integers has been introduced by P. Moree in $1997$  together with  the characterization of  good odd integers. Such integers have shown to have nice number theoretical properties  and wide applications. In this paper, a complete characterization of all good integers is given. 
Two subclasses of   good integers are introduced, namely, oddly-good and evenly-good integers. The characterization and properties of  good   integers in these two subclasses  are determined.
As applications,  good integers  and oddly-good integers are applied  in the study of  the hulls of  abelian codes.  The average dimension of   the  hulls of abelian codes  is given  together with some upper and lower bounds.    

We note that   the published version \cite{J2017}  contains some errors in  \cite[Proposition 2.1]{J2017} and \cite[Proposition 2.3]{J2017}.    These results have been corrected and updated in this manuscript. The correction  does not affect any other part of \cite{J2017}.
\end{abstract}

\noindent Keywords: {Good integers; Abelian  codes;  Hull of abelian codes; Euclidean inner product; Hermitian inner product.}

\noindent  Mathematics Subject Classification 2010: 11N25, 11B83, 94B15, 94B60

\section{Introduction}

The concept of good integers has been introduced in \cite{M1997} by P. Moree.  For fixed  coprime nonzero integers $a$ and $b$,  a positive  integer $d$ is said to be {\em good (with respect to $a$ and $b$)} if  it is a divisor of $a^k+b^k$ for some integer $k\geq 1$. Otherwise, $d$ is said to be {\em bad}.     Denote by $G_{(a,b)}$   the set of good   integers defined with respect to $a$ and $b$.  The  characterization of odd integers in  $G_{(a,b)}$ has been given in \cite{M1997}.   In \cite{KG1969}, the set $G_{(q,1)}$ has been  studied  and  applied in constructing BCH codes with good design distances, where $q$ is a prime power. The set   $G_{(2^\nu ,1)}$  has been applied in counting the Euclidean self-dual  cyclic and abelian codes over finite fields in  \cite{JLX2011} and \cite{JLLX2012}, respectively.

In this paper, two subclasses of good integers are introduced as follows.  For given nonzero coprime integers  $a$ and $b$, a positive integer $d$ is said to be {\em oddly-good (with respect to $a$ and $b$)} if $d$ divides $a^k+b^k$ for some odd integer $k\geq 1$, and {\em evenly-good (with respect to $a$ and $b$)} if $d$ divides $a^k+b^k$ for some even integer $k\geq 2$. Therefore, $d$ is   { good} if it is oddly-good or evenly-good. Denote by $OG_{(a,b)}$ (resp.,    $EG_{(a,b)}$) the set of oddly-good (resp., evenly-good) integers defined with respect to $a$ and $b$.  Clearly, $G_{(a,b)}=G_{(b,a)} $, $OG_{(a,b)}=OG_{(b,a)} $ and $EG_{(a,b)}=EG_{(b,a)} $.  In \cite{JLS2014}, some basic properties of  $OG_{(2^\nu,1)}$ and $EG_{(2^\nu ,1)}$ have been studied and applied in enumerating Hermitian self-dual abelian  codes over finite fields.   

The hull of a linear code, the intersection of a code and it dual, is key to determine   the complexity  of algorithms    in determining  the automorphism group of a linear code and testing the permutation  equivalence of  two codes (see \cite{Sendrier997}, \cite{Leon82}, \cite{Leon91}, \cite{Leon97}, \cite{Sendrier00} and \cite{Sendrier01}).   Precisely, most of the algorithms do  not work if the  size of the hull  is large.   In \cite{Sendrier97}, the number of distinct linear codes of length $n$ over  $\mathbb{F}_q$ which have  hulls of a given dimension has been established.  
In \cite{S2003},   some additional properties of $G_{(q,1)}$  have been studied and applied in the determination of the average dimension of the hulls of cyclic codes.  In \cite{DMS2007} and \cite{DMS2014},  $G_{(q,1)}$ has been applied in the study of some Hermitian self-dual codes.  Later, in \cite{SJLP2015},   $G_{(q,1)}$ has been  applied in  determining the dimensions of the hulls of cyclic and negacyclic codes over finite fields.    To the best of our knowledge,  properties of the hulls of  abelian codes have not been well studied.

In this paper, we aim to   characterize  the classes of good integers, oddly-good integers and evenly-good integers defined with respect to arbitrary coprime   nonzero  integers $a $ and $b$.  
As applications,   the hulls of abelian codes are studied using good and oddly-good integers. The   average dimension of the hulls of abelian codes is determined under both the Euclidean and Hermitian inner products.

The paper is organized as follows. In Section $2$, some properties of good odd integers are recalled and the  characterization of all good integers is given. The characterizations of oddly-good and evenly-good  integers
are given in Section $3$.  In Section $4$, 
applications of good integers in  determining the  average dimension of   the  hulls of abelian codes are given.  The summary is given in Section~5.

\section{Good Integers}

In this section,  some basic properties of  good odd integers in \cite{M1997}  are recalled and the characterization of  arbitrary good integers is given.

For  pairwise coprime nonzero  integers $a,b$ and $n>0$, let ${\rm ord}_n(a)$ denote the multiplicative order $a$ modulo $n$. In this case, $b^{-1} $ exists in the multiplicative group $\mathbb{Z}_n^\times$.  Let ${\rm ord}_n(\frac{a}{b})$ denote the multiplicative order $ab^{-1}$ modulo $n$.  Denote by $2^\gamma ||n$ if $\gamma $ is the largest integer such that $2^\gamma |n$. 

From the definition, $1$ is always good and we have the following property.
\begin{lemma}
    \label{odd-good2} Let $a$ and $ b$  be nonzero coprime  integers and let $d$ be a positive    integer. If    $d\in G_{(a,b)}$, then     $\gcd(a,d)=1=\gcd(b,d)$.  
\end{lemma}
\begin{proof}
    Assume that   $d\in G_{(a,b)}$. Suppose that $\gcd(a,d)\ne 1$ or $\gcd(b,d)\ne 1$.   Since  $G_{(a,b)}=G_{(b,a)}$,  we may assume that  $p|\gcd(a,d)$ for some prime $p$.  Then there exists a positive integer $k$ such that  $d|(a^k+b^k)$   which implies that $p|b^k$.   It follows that $p|\gcd(a,b)$, a contradiction.
\end{proof}

Properties of good odd integers have been studied in \cite{M1997}. Some results used in this paper  are summarized as follows.

\begin{lemma}
    [{\cite[Proposition 2]{M1997}}]\label{2order} Let $p$ be an odd prime and let $r$ be a positive integer. If $p^r$ is good, then ${\rm ord}_{p^r}(\frac{a}{b})=2s$, where $s$ is the smallest positive integer such that $({a}{b}^{-1})^s\equiv -1 \,({\rm mod}\, p^{r})$. 
\end{lemma}
\begin{lemma}
    [{\cite[Proposition 4]{M1997}}]\label{ord} Let $p$ be an odd prime and let $r$ be a positive integer. Then ${\rm ord}_{p^r}(\frac{a}{b})={\rm ord}_{p}(\frac{a}{b})p^i$ for some $i\geq 0$. 
\end{lemma}

\begin{lemma}
    [{\cite[Theorem 1]{M1997}}]\label{goodP} Let $d>1$ be an odd integer. Then $d\in G_{(a,b)}$   if and only if there exists $s\geq 1$ such that $2^s|| {\rm ord}_{p}(\frac{a}{b})$ for every prime $p$ dividing~$d$. 
\end{lemma}

It has been briefly discussed in \cite{M1997} that a good even integer exists if   and only  if $ab$ is odd.  For completeness, the characterization and properties of all good (odd and even) integers  are given as follows.

(The following proposition is a corrected version of    {\cite[Proposition~2.1 ]{J2017}}.)
\begin{proposition} \label{evengood} Let $a$ and $b$   be  coprime odd  integers  and  let $\beta\geq 1$ be  an  integer. Then the following statements are equivalents.
    \begin{enumerate}[$1)$]
        \item
        $2^\beta\in G_{(a,b)}$.
        \item      $2^\beta|(a+b)$.
        \item  $ab^{-1} \equiv -1 ~{\rm mod}~ 2^\beta$. 
    \end{enumerate}
\end{proposition}
\begin{proof}
    To prove 1) implies 2), assume that   $2^\beta \in G_{(a,b)}$.  If  $\beta=1$, then   $2^\beta|(a+b)$   since $a+b$ is even. Then  $2^\beta|(a^k+b^k)$ for some integer $k\geq 1$.
    Assume that $\beta>1$. Then $4|(a^k+b^k)$. If $k$ is even, then    $a^k\equiv 1 \ {\rm mod\ } 4$ and $b^k\equiv 1 \ {\rm mod\ } 4$  which implies that $(a^k+b^k)\equiv 2\, {\rm mod }\, 4$, a contradiction. 
    It follows that   $k$ is odd.  Since $a^k+b^k=(a+b)\left(\sum\limits_{i=0}^{k-1} (-1)^ia^{k-1-i}b^i \right)$ and  $\sum\limits_{i=0}^{k-1} (-1)^ia^{k-1-i}b^i $ is odd, we have that  $2^\beta |(a+b)$.  
    
    The statement  2) $\Rightarrow$  1)  follows from the definition. The equivalent statement  2) $\Leftrightarrow$ 3) is obvious.
\end{proof}

\begin{proposition} \label{prop2d}
    Let $a$, $b$  and $d>1$ be  pairwise coprime odd     integers.  Then    $d\in G_{(a,b)}$
    if and only if        $2d\in G_{(a,b)}$. 
    In this case, 	   ${\rm ord}_{2d}(\frac{a}{b})= {\rm ord}_{d}(\frac{a}{b})$ is even. 
\end{proposition}
\begin{proof}  Since $d$ is odd and $a^k+b^k$ is even for all  positive integers $k$,   $d|(a^k+b^k) $ if and only if $2d|(a^k+b^k) $.   The characterization follows immediately.
    
    Assume that $d\in  G_{(a,b)}$. Let  $k$  be the smallest positive integer such that $d|(a^k+b^k) $.  It follows that  $ (ab^{-1})^k\equiv -1\ {\rm mod\ } d$  and ${\rm ord}_{d}(\frac{a}{b})\nmid k$. Hence, ${\rm ord}_{d}(\frac{a}{b})=2k$. 
    Since $a^k+b^k$ is even and $d$ is odd, we have $2d|(a^k+b^k)$. Hence, $ (ab^{-1})^k\equiv -1\ {\rm mod\ } 2d$ which implies that  ${\rm ord}_{2d}(\frac{a}{b})|2k$.   Since ${\rm ord}_{d}(\frac{a}{b}) \leq {\rm ord}_{2d}(\frac{a}{b})$, we have  that  ${\rm ord}_{2d}(\frac{a}{b}) = {\rm ord}_{2d}(\frac{a}{b})=2k $   is even. 
    %
\end{proof}

(The next proposition is a  correction of   \cite[Proposition 2.3]{J2017}.)
\begin{proposition}  \label{prop2m} Let $a,b$ and $d>1$   be pairwise coprime odd positive integers  and let  $\beta\geq 2$ be  an  integer. Then    $2^\beta d\in G_{(a,b)}$ if and only if   $2^\beta|(a+b)$
    and  $d\in G_{(a,b)}$  is such that  $2|| {\rm ord}_{d}(\frac{a}{b})$.  
    In this case,    ${\rm ord}_{2^\beta}(\frac{a}{b})=2$  and   $2|| {\rm ord}_{2^\beta d}(\frac{a}{b})$.
\end{proposition}
\begin{proof} 
    Assume that   $2^\beta d\in G_{(a,b)}$.  
    Let $k$ be the smallest   positive integer  such that $2^\beta d |(a^k+b^k)$.  Then    $d|(a^k+b^k) $ and  $2^\beta |(a^k+b^k)$   which implies that  $d\in G_{(a,b)}$ and $(ab^{-1})^{2k} \equiv 1 ~{\rm mod} ~d$.  Moreover,   $2^\beta |(a+b)$   and    $k$ must be odd   by    Proposition~\ref{prop2.01} and its proof.
    Let $k'$ be the smallest  positive integer  such that  $d|(a^{k'}+b^{k'})$.  Then $ {\rm ord}_{d}(\frac{a}{b})=2k'$.  Since  $(ab^{-1})^{2k} \equiv 1 ~{\rm mod} ~d$,  we have  $k'|k$.    Consequently,  $k'$ is odd   and    $(a+b)|(a^{k'}+b^{k'})$.  Hence, $2^\beta d| (a^{k'}+b^{k'})$. By the minimality of $k$, we have  $k= k'$ and  $d|(a^k+b^k)$. Consequently, ${\rm ord}_{d}(\frac{a}{b})=2k'=2k$.
    Since  $k$ is odd,   $d\in G_{(a,b)}$ is such that  $2|| {\rm ord}_{d}(\frac{a}{b})$.

    Conversely, assume that   $2^\beta|(a+b)$ and $d\in G_{(a,b)}$ is such that  $2|| {\rm ord}_{d}(\frac{a}{b})$.   Let $k$ be the   smallest positive integer  such that  $d|(a^k+b^k)$. Then $ (ab^{-1})^{k}\equiv -1\ {\rm mod\ } d$ which implies that $ {\rm ord}_{d}(\frac{a}{b})=2k$.  Since $2|| {\rm ord}_{d}(\frac{a}{b})$, $k$ must be  odd.  It follows that    $  (ab^{-1})^{k}\equiv ab^{-1}\equiv -1\ {\rm mod\ } 2^\beta$.  Since $d$ is odd,  $ (ab^{-1})^{k}\equiv -1\ {\rm mod\ } 2^\beta d$.
    Hence,  $2^\beta d|(a^{k}+b^k)$ which means  $2^\beta d\in  G_{(a,b)} $ as desired.

    In this case,      we have $2^\beta|(a+b)$ which implies that
    ${\rm ord}_{2^\beta}(\frac{a}{b})=2$. Moreover, $ {\rm ord}_{2^\beta  d}(\frac{a}{b})= {\rm lcm} \left({\rm ord}_{2^\beta  }(\frac{a}{b}), {\rm ord}_{ d}(\frac{a}{b}) \right)=2k $ and $k$ is odd. Therefore, 
    $2|| {\rm ord}_{2^\beta  d}(\frac{a}{b})$.  
\end{proof}

From Propositions \ref{evengood}--\ref{prop2m}, the characterization of good integers  can be summarized   based on $\beta$ and the parity of $ab$   as follows.
\begin{theorem} \label{thmChar} Let $a$ and $b$    be     coprime  nonzero   integers and let $\ell=2^\beta d$ be a positive integer such that $d$ is odd and $\beta \ge 0$.  Then one of the following statements holds.
    \begin{enumerate}
        \item  If  $ab$ is odd, then $\ell=2^\beta d\in G_{(a,b)}$  if and only if  one of the following statements holds.
        \begin{enumerate}
            \item $\beta \in\{0,1\}$ and  $d=1$.
            \item $\beta \in \{0,1\}$, $d\geq 3$   and  
            there exists $s\geq 1$ such that $2^s|| {\rm ord}_{p}(\frac{a}{b})$ for every prime $p$ dividing $d$. 
           \item {$\beta \ge 2$, $d=1$  and    $ 2^\beta |(a+b)$.}
           \item $\beta \ge  2$, $d\geq 3$,     $ 2^\beta |(a+b)$  and  $ d\in G_{(a,b)}$ is such that $2|| {\rm ord}_{d}(\frac{a}{b})$.   
        \end{enumerate}
        \item  If $ab$ is even,  then $\ell=2^\beta d\in G_{(a,b)}$  if and only if    one of the following statements holds.
        \begin{enumerate}
            \item $\beta =0$   and $d=1$.
            \item  $\beta =0$, $d\geq 3$,   and 
            there exists $s\geq 1$ such that $2^s|| {\rm ord}_{p}(\frac{a}{b})$ for every prime $p$ dividing $d$. 
        \end{enumerate}
    \end{enumerate}
\end{theorem}
We note that the condition  ${\rm ord}_{2^\beta }(\frac{a}{b})=2$  in  Theorem \ref{thmChar} is equivalent to $2^\beta|(a+b)$ by Proposition \ref{evengood}.

From Propositions \ref{evengood}--\ref{prop2m} and Theorem \ref{thmChar}, 
other necessary conditions for  a positive integer  to be good    can be   given in the following corollary. These are   sometime useful in  applications.

\begin{corollary}\label{cor:good-gamma}
    Let $a$ and $b$    be    coprime  nonzero  integers and let $\ell=2^\beta d$ be a positive integer such that $d$ is odd and $\beta \ge 0$. Let $\gamma\geq 0$ be an integer such that $2^\gamma|(a+b)$.  If  $\ell\in G_{(a,b)}$, then  one of the following statements holds.
    \begin{enumerate}
        \item {$\ell=1$ and   $ab$ is even.}
        \item {$\ell\in \{1,2\}$ and   $ab$ is odd.}
        \item $2|| {\rm ord}_{\ell}(\frac{a}{b})$ if and only if  one of the following statements holds.
        \begin{enumerate}
            \item $d=1$ and $2\leq \beta \le \gamma$.
            \item $d\geq 3$, $2|| {\rm ord}_{p}(\frac{a}{b})$ for every prime $p$ dividing $d$, and $0\leq \beta  \leq \gamma$. 
        \end{enumerate}
        \item  There exists an integer $s\geq 2$ such that  $2^s|| {\rm ord}_{\ell}(\frac{a}{b})$ if and only if  $d\geq 3$, $2^s|| {\rm ord}_{p}(\frac{a}{b})$ for every prime $p$ dividing $d$, and $0\leq \beta  \leq 1$. 
    \end{enumerate}
\end{corollary}


\section{Oddly-Good and Evenly-Good Integers}

For given   pairwise coprime  nonzero   integers $a$, $b$ and $\ell>0$,  recall that $\ell$  is said to be {\em oddly-good} if $\ell$ divides $a^k+b^k$ for some odd integer $k\geq 1$, and {\em evenly-good} if $\ell$ divides $a^k+b^k$ for some even integer $k\geq 2$. In this section, the characterizations and properties of oddly-good and evenly-good integers are determined.

We note that $1$ is always good.  Since $1| \left(a+b\right)$ and  $1|\left (a^2+b^2\right)$, we have that $1$ is both oddly-good and evenly-good.   The integer $2$ is good if and only if $ab$ is odd. In this case, $a+b$ is even and hence,  $2|(a+b)$ and $2|(a^2+b^2)$ which imply that $2$ is both oddly-good and evenly-good. However, for an  integer $\ell>2$,  $\ell$ can be either oddly-good or evenly-good, but not both. 
\begin{proposition}
    \label{disj} Let $a$, $b$ and $\ell>2$ be   pairwise coprime  nonzero integers. If   $\ell\in G_{(a,b)}$, then either $\ell \in OG_{(a,b)} $ or $\ell \in EG_{(a,b)}$, but not both. 
\end{proposition}
\begin{proof} We distinguish the proof into four cases. 
    \\ {\bf Case 1} $\ell$ is odd. Assume that $s$ and $t$ are positive integers such that $s\geq t$, $\ell| (a^s+b^s)$, and $\ell | (a^t+b^t)$. It is sufficient to show that $s$ and $t$ have the same parity. Let $p$ be a prime divisor of $\ell$. Then $p| (a^s+b^s)$ and $p| (a^t+b^t)$, or equivalently, $a^s\equiv -b^s \, {\rm mod}\, p $ and $a^t\equiv -b^t \, {\rm mod}\, p $.  These hold true if and only if  $(ab^{-1})^s\equiv -1 \, {\rm mod}\, p $ and $(ab^{-1})^t\equiv -1 \, {\rm mod}\, p $.
    It follows that $(ab^{-1})^s-(ab^{-1})^t\equiv 0 \, {\rm mod}\, p    $, and hence,  $(ab^{-1})^t((ab^{-1})^{s-t}-1)\equiv 0 \, {\rm mod}\, p    $.  This means $(ab^{-1})^{s-t}\equiv 1 \, {\rm mod}\, p $ which implies that   ${\rm ord}_p(   \frac{a}{b})\mid (s-t)$. Since $d$ is good, by Lemma~\ref{goodP}, ${\rm ord}_p(\frac{a}{b})$ is even which implies $s-t$ is even. Therefore, $s$ and $t$ have the same parity. 
    \\ {\bf Case 2} $\ell=2^\beta $, where $\beta \ge 2$.  Then $\ell$ is oddly-good by the proof of Proposition~\ref{evengood}. 
    \\ {\bf Case 3} $\ell=2d$, where $d$ is odd.  By Proposition \ref{prop2d}, $d\in G_{(a,b)}$.  Moreover, $d$ is either oddly-good or evenly-good by Case 1. For each positive integer $k$,   $d|(a^k+b^k) $ if and only if $2d|(a^k+b^k) $. Therefore, $2d$ is oddly (resp., evenly)  good if and only if $d$ is oddly (resp., evenly)  good. 
    \\ {\bf Case 4} $\ell=2^\beta d$, where $d$ is odd and $\beta \ge 2$.  By the proof of Proposition \ref{prop2m},  $\ell$ is oddly-good.
\end{proof}

From Proposition \ref{disj}, it follows that $G_{(a,b)}=OG_{(a,b)}\cup EG_{(a,b)}$. Moreover, we have  $OG_{(a,b)}\cap EG_{(a,b)}=\{1\}$ if $ab$ is even and  $OG_{(a,b)}\cap EG_{(a,b)}=\{1,2\}$ if $ab$ is odd.  Properties of   $OG_{(a,b)}$ and $EG_{(a,b)}$
are determined in the next subsections.





%

\subsection{Oddly-Good Integers}
In this subsection,  we focus on the characterization and properties of oddly-good integers.

The following lemma from \cite{M1997} is useful.
\begin{lemma}
    [{\cite[Lemma 1]{M1997}}]\label{congr} Let $a_1,a_2,\dots,a_t$ be positive integers. Then the system of congruences 
    \begin{align*}
    x\equiv a_1\,({\rm mod}\, 2a_1), \quad x\equiv a_2\,({\rm mod}\, 2a_2), \quad \dots, \quad x\equiv a_t\,({\rm mod}\, 2a_t) 
    \end{align*}
    has a solution $x$ if and only if there exists $s\geq 0$ such that $2^s|| a_i$ for all $i=1,2,\dots, t$. 
\end{lemma}

\begin{proposition}
    \label{odd-good} Let $a$ and $b$ be coprime nonzero integers and let $d>1$ be an  odd  integer. Then   $d\in OG_{(a,b)}$   if and only if $2|| {\rm ord}_{p}(\frac{a}{b})$ for every prime $p$ dividing $d$.  
\end{proposition}
\begin{proof}
    Let $p_1,p_2,\dots,p_t$ be the distinct prime divisors of $d$. For each $1\leq i\leq t$, let $r_i$ be the positive integer such that $p_i^{r_i}|| d$. To prove the necessity, assume that $d\in OG_{(a,b)}$. There exists an odd positive integer $c$ such that $(\frac{a}{b})^c\equiv -1 \,({\rm mod}\, d)$ and, hence, $(\frac{a}{b})^c\equiv -1 \,({\rm mod}\, p_i^{r_i})$ for all $i=1,2,\dots,t$. By Lemma~\ref{2order}, it follows that ${\rm ord}_{p_i^{r_i}}(\frac{a}{b})$ is even and
    $ c\equiv \frac{{\rm ord}_{p_i^{r_i}}(\frac{a}{b})}{2} \,({\rm mod}\, {\rm ord}_{p_i^{r_i}}(\frac{a}{b})) $
    for all $i=1,2,\dots,t$. Since $c$ is an odd integer, $\frac{{\rm ord}_{p_i^{r_i}}(\frac{a}{b})}{2}$ is odd for all $i=1,2,\dots,t$. Therefore, by Lemma~\ref{ord}, $2|| {\rm ord}_{p}(\frac{a}{b})$ for every prime $p$ dividing $d$.
    
    Conversely, assume that $2|| {\rm ord}_{p}(\frac{a}{b})$ for every prime $p$ dividing $d$. Then, by Lemma~\ref{ord}, we have $2|| {\rm ord}_{p_i^{r_i}}(\frac{a}{b})$ for all $i=1,2,\dots,t$. By Lemma~\ref{congr}, there exists an integer $c$ such that
    $c\equiv \frac{{\rm ord}_{p_i^{r_i}}(\frac{a}{b})}{2} \,({\rm mod}\, {\rm ord}_{p_i^{r_i}}(\frac{a}{b})) $
    for all $i=1,2,\dots,t$. Since $\frac{{\rm ord}_{p_i^{r_i}}(\frac{a}{b})}{2}$ is odd, $c$ is an odd integer. Thus $(\frac{a}{b})^c\equiv -1 \,({\rm mod}\, p_i^{r_i})$ for all $i=1,2,\dots,t$, and hence, $(\frac{a}{b})^c\equiv -1 \,({\rm mod}\, d)$. Therefore, $d$ is oddly-good as desired.
\end{proof}

From Proposition \ref{odd-good},  we have the following characterization.
\begin{corollary}
    \label{cor-oddly} Let $a$ and $b$ be coprime nonzero integers and let $d>1$ be an  odd  integer. Then the following statements are equivalent. 
    \begin{enumerate} 
        \item $d\in OG_{(a,b)}$. 
        \item $j\in OG_{(a,b)}$ for all divisors $j$ of $d$. 
        \item $p\in OG_{(a,b)}$  for all prime divisors $p$ of $d$. 
    \end{enumerate}
\end{corollary}
\begin{proof}
    The statements $(1)\Rightarrow (2)$ and  $(2)\Rightarrow (3)$ are obvious. By Proposition~\ref{odd-good},  the statement $(3) \Rightarrow (1)$ follows.
\end{proof}

From the proof of Propositions \ref{prop2d}, \ref{prop2m} and  \ref{disj},  we have the following characterizations.
\begin{corollary}
    \label{odd-good2} Let $a$ and $b$ be coprime nonzero integers and let $d>1$ be an  odd  integer.  
    \begin{enumerate}
        \item  The following statements are equivalent.
        \begin{enumerate}
            \item $d\in OG_{(a,b)}$.
            \item       $2d\in OG_{(a,b)}$. 
        \end{enumerate} 
        \item  For each $\beta \ge 2$,  $2^\beta d\in OG_{(a,b)}$ if and only if $2^\beta d\in G_{(a,b)}$.
    \end{enumerate}    
\end{corollary}

The characterization  of oddly-good integers can be  summarized in the following theorem.
\begin{theorem} \label{oddlygood}Let $a$ and $b$     be   coprime  nonzero integers  and  let $\ell =2^\beta d$ be  an  integer such that $d$ is odd and $\beta \ge 0$. Then  one of the following statements holds.
    \begin{enumerate}
        \item  If  $ab$ is odd, then $\ell=2^\beta d\in OG_{(a,b)}$  if and only if  one of the following statements holds.
        \begin{enumerate}
            \item $\beta \in \{0,1\}$ and  $d=1$.
            \item
            $\beta \in \{0,1\}$, $d\geq 3$,   and  $2|| {\rm ord}_{p}(\frac{a}{b})$ for every prime $p$ dividing $d$. 
         \item {$\beta \ge 2$, $d=1$  and    $ 2^\beta |(a+b)$.}
         \item $\beta \ge  2$, $d\geq 3$,     $ 2^\beta |(a+b)$  and  $ d\in G_{(a,b)}$ is such that $2|| {\rm ord}_{d}(\frac{a}{b})$.   
        \end{enumerate}
        \item  If $ab$ is even,  then $\ell=2^\beta d\in OG_{(a,b)}$  if and only if    one of the following statements holds.
        \begin{enumerate}
            \item  $\beta =0$   and $d=1$.
            \item 
            $\beta =0$, $d\geq 3$,   and  $2|| {\rm ord}_{p}(\frac{a}{b})$ for every prime $p$ dividing $d$. 
        \end{enumerate}
    \end{enumerate}
    
\end{theorem}
Note that the condition  ${\rm ord}_{2^\beta }(\frac{a}{b})=2$  in  Theorem \ref{oddlygood} is equivalent to $2^\beta|(a+b)$ by Proposition \ref{evengood}.

The following necessary conditions for a  positive integer to be oddly-good can be concluded from Proposition~\ref{odd-good} and  Theorem \ref{oddlygood}. 
\begin{corollary}
    Let $a$ and $b$    be   coprime  nonzero  integers and let $\ell=2^\beta d$ be a positive integer such that $d$ is odd and $\beta \ge 0$. Let $\gamma\geq 0$ be an integer such that $2^\gamma||(a+b)$.  If  $\ell\in OG_{(a,b)}$, then one of the following statements holds.
    \begin{enumerate}
        \item {$\ell=1$ and  $ab$ is even.}
        \item {$\ell\in \{1,2\}$ and  $ab$ is odd. }
        \item $2|| {\rm ord}_{\ell}(\frac{a}{b})$ if and only if  one of the following statements holds.
        \begin{enumerate}
            \item $d=1$ and $2\leq \beta \le \gamma$.
            \item $d\geq 3$, $2|| {\rm ord}_{p}(\frac{a}{b})$ for every prime $p$ dividing $d$, and $0\leq \beta  \leq \gamma$. 
        \end{enumerate}
    \end{enumerate}
\end{corollary}

\subsection{Evenly-Good Integers}
In this subsection,  we focus on properties of evenly-good integers. From Proposition \ref{disj},  an integer greater than $2$ can be either oddly-good or evenly-good. Then, by Lemma \ref{goodP} and Proposition \ref{odd-good}, we have the following characterization of evenly-good integers.

\begin{proposition}
    \label{odd-even} Let $a$ and $b$ be coprime nonzero  integers and let $d>1$ be an  odd  integer. Then  
    $d\in EG_{(a,b)}$ is    if and only if there exists $s\geq 2$ such that $2^s|| {\rm ord}_{p}(\frac{a}{b})$ for every prime $p$ dividing $d$.  
\end{proposition}

\begin{corollary}
    \label{cor-evenly} Let $a$ and $b$ be coprime nonzero integers and let $d>1$ be an  odd  integer. If $d\in EG_{(a,b)}$   (resp. $d\in G_{(a,b)}$), then $j\in EG_{(a,b)}$   (resp. $j\in G_{(a,b)}$)  for all divisors $j$ of $d$. 
    
\end{corollary}

The above properties of evenly-good integers can be summarized as follows.

\begin{theorem}\label{thm-even} Let $a$ and $b$     be   coprime   nonzero  integers  and  let $\ell =2^\beta d$ be  an  integer such that $d$ is odd and $\beta \ge 0$. Then   one of the following statements holds.
    \begin{enumerate}
        \item  If  $ab$ is odd, then $\ell=2^\beta d\in EG_{(a,b)}$  if and only if  one of the following statements holds.
        \begin{enumerate}
            \item $\beta \in \{0,1\}$ and  $d=1$.
            \item $\beta \in \{0,1\}$, $d\geq 3$,  and    there exists $s\geq 2$ such that $2^s|| {\rm ord}_{p}(\frac{a}{b})$ for every prime $p$ dividing $d$.  
        \end{enumerate}
        \item  If $ab$ is even,  then $\ell=2^\beta d\in EG_{(a,b)}$  if and only if    one of the following statements holds.
        \begin{enumerate}
            \item  $\beta =0$ and  $d=1$.
            \item  $\beta =0$, $d\geq 3$,    and 
            there exists $s\geq 2$ such that $2^s|| {\rm ord}_{p}(\frac{a}{b})$ for every prime $p$ dividing $d$.  
        \end{enumerate}
    \end{enumerate}
    
\end{theorem}

The following necessary conditions for a positive integer to be  evenly-good integers  can be obtained from Proposition \ref{odd-even} and Theorem \ref{thm-even}.

\begin{corollary}
    Let $a$ and $b$    be   coprime  nonzero  integers and let $\ell=2^\beta d$ be a positive integer such that $d$ is odd and $\beta \ge 0$. Let $\gamma\geq 0$ be an integer such that $2^\gamma||(a+b)$. If  $\ell\in EG_{(a,b)}$, then one of the following statements holds.
    \begin{enumerate}
        \item {$\ell=1$ and  $ab$ is even.}
        \item {$\ell\in \{1,2\}$ and $ab$ is odd.}
        \item There exists an integer $s\geq 2$ such that  $2^s|| {\rm ord}_{\ell}(\frac{a}{b})$ if and only if  $d\geq 3$, $2^s|| {\rm ord}_{p}(\frac{a}{b})$ for every prime $p$ dividing $d$, and $0\leq \beta  \leq 1$.
    \end{enumerate}
\end{corollary}

%

\section{Applications}
This section is devoted to applications of good integers in  coding theory.     As discussed in the introduction, good and oddly-good integers have wide applications in coding theory. The important ones are  the construction of good BCH codes in
\cite{KG1969}, the enumerations of  the Euclidean self-dual  cyclic and abelian codes   in  \cite{JLX2011} and \cite{JLLX2012}, the enumeration of  Hermitian  self-dual    abelian codes in 
\cite{JLS2014},   the study of the average dimension of the hulls of cyclic codes in 
\cite{S2003},  and the  determinations of  the dimension of the hulls of cyclic and negacyclic codes in \cite{SJLP2015}.

The hull of a linear code  is known to be     key  in determining  the complexity  of algorithms    in finding  the automorphism group of a linear code and testing the permutation  equivalence of  two codes \cite{Sendrier997},  \cite{Leon82},  \cite{Leon91},  \cite{Leon97}, \cite{Sendrier00},  and \cite{Sendrier01}.    Therefore, the study of the hulls is one of the interesting problems in coding theory. Some properties of the hulls of linear and cyclic codes have been further studied    in \cite{Sendrier97}, \cite{S2003} and \cite{SJLP2015}.

In this section, we  focus on applications of $G_{(p^\nu,1)}$ and $OG_{(p^\nu,1)}$ in determining  the average dimension of the  hulls of abelian codes under both the Euclidean and Hermitian inner products which have not been well studied. Specifically, we are mainly focused on the hulls of abelian codes  in principal ideal group  algebras. As a special case, the results   on  the hulls of cyclic codes in \cite{S2003} can be viewed as corollaries.

\subsection{Abelian Codes in Principal Ideal Group Algebras}

For  a prime $p$ and a positive integer $\nu$,    let  $\mathbb{F}_{p^\nu}$ denote   the  finite field of $p^\nu$.  Let $G$ be a finite abelian group, written additively. Denote by    $\mathbb{F}_{p^\nu}[G]$   the {\it group algebra} of
$G$ over~$\mathbb{F}_{p^\nu}$. The elements in $ \mathbb{F}_{p^\nu}[G]$ will be written as $\sum_{g\in G}\alpha_{{g }}Y^g $,
where $ \alpha_{g }\in \mathbb{F}_{p^\nu}$.  The addition and the multiplication in $ \mathbb{F}_{p^\nu}[G]$ are  given as in the usual polynomial rings over $\mathbb{F}_{p^\nu}$ with the indeterminate $Y$, where the indices are computed additively in $G$.      A group  algebra  $ \mathbb{F}_{p^\nu}[G]$ is said to be a  {\em principal ideal group  algebra} (PIGA) if every ideal in $ \mathbb{F}_{p^\nu}[G]$ is generated by a single element. 
In \cite{FiSe1976}, it has been shown that $\mathbb{F}_{p^\nu}[G]$ is a PIGA if and only if the Sylow $p$-subgroup of $G$  is cyclic.

An {\em abelian code} in  $\mathbb{F}_{p^\nu}[G]$  is   an ideal in $\mathbb{F}_{p^\nu}[G]$ (see \cite{JLLX2012} and \cite{JLS2014}). The {\em Euclidean inner product}  between $u=\sum_{g\in G}u_{{g }}Y^g $ and  $v=\sum_{g\in G}v_{{g }}Y^g $   $\mathbb{F}_{p^\nu}[G]$  is defined to be 
$\langle u,v\rangle _{\rm E}:= \sum_{g\in G} u_gv_g.$ The {\em Euclidean dual} of an abelian code $C$ in  $\mathbb{F}_{p^\nu}[G]$ is defined to be $C^{\perp_{\rm E}}:=\{v\in \mathbb{F}_{p^\nu}[G]\mid  \langle c,v\rangle _{\rm E} =0 \text{ for all } c\in C\}$. The {\em Euclidean hull} of a code $C$ is defined to be  $\mathcal{H}_{\rm E}(C):=C\cap C^{\perp_{\rm E}}.$
In $\mathbb{F}_{{p^{2\nu}}}[G]$, the {\em Hermitian inner product}  between $u=\sum_{g\in G}u_{{g }}Y^g $ and  $v=\sum_{g\in G}v_{{g }}Y^g $   $\mathbb{F}_{p^{2\nu}}[G]$  is defined to be 
$\langle u,v\rangle _{\rm H}:= \sum_{g\in G} u_gv_g^{p^\nu}.$ The {\em Hermitian dual} of an abelian code $C$ in  $\mathbb{F}_{p^{2\nu}}[G]$ is defined to be $C^{\perp_{\rm H}}:=\{v\in \mathbb{F}_{{p^{2\nu}}}[G]\mid  \langle c,v\rangle _{\rm H} =0 \text{ for all } c\in C\}$. The {\em Hermitian hull} of a code $C$ is defined to be  $\mathcal{H}_{\rm H}(C):=C\cap C^{\perp_{\rm H}}.$

The average dimension of the Euclidean (resp.  Hermitian) hulls of abelian codes in  $\mathbb{F}_{p^{\nu}}[G]$ (resp.,  in  $\mathbb{F}_{{p^{2\nu}}}[G])$ is defined to be 
\[   {\rm avg}_{p^\nu}^{\rm E}(G):=\frac{\sum\limits_{C\in \mathcal{C}(p^\nu,G)}\dim(\mathcal{H}_{\rm E}(C))}{|\mathcal{C}(p^\nu,G)|}\text{ (resp., $ {\rm avg}_{{p^{2\nu}}}^{\rm H}(G):=\frac{\sum\limits_{C\in \mathcal{C}({p^{2\nu}},G)}\dim(\mathcal{H}_{\rm H}(C))}{|\mathcal{C}({p^{2\nu}},G)|} $)},\]
where $\mathcal{C}(p^\nu,G)$ (resp., in  $\mathcal{C}({p^{2\nu}},G)$) is the set of all abelian codes in  $\mathbb{F}_{p^{\nu}}[G]$ (resp., in  $\mathbb{F}_{{p^{2\nu}}}[G])$.

The rest of this section,  we focus on  abelian codes in    PIGAs.  It therefore suffices to restrict the study to  a group algebra  $\mathbb{F}_{p^\nu}[ A\times \mathbb{Z}_{p^k}]$, where  $A$ is a finite  abelian group  such that $p\not | |A|$ and  $k\geq 0$ is an integer.

For  positive integers $i$ and $j$ with $\gcd(i,j)=1$, let ${\rm ord}_j(i)$ denote the multiplicative order of $i$ modulo $j$. For each $a\in A$, denote by ${\rm ord}(a)$ the additive order of $a$ in $A$.  For a positive integer $q$ with $\gcd(|A|,q)=1$, a {\it $q$-cyclotomic class}   of $A$ containing $a\in A$ is defined to be the set
\begin{align*}
S_{q}(a):=&\{q^{ i}\cdot a \mid i=0,1,\dots\}
=\{q^{ i}\cdot a \mid 0\leq i< {\rm ord}_{{\rm ord}(a)}(q) \}, 
\end{align*}
where $q^{i}\cdot a:= \sum\limits_{j=1}^{q^{i}}a$ in $A$.

First, we consider the decomposition of  $\mathcal{R}:=\mathbb{F}_{p^{\nu}}[A]$.  In this case, $\mathcal{R}$ is semi-simple (see \cite{Be1967_2}) which can be  decomposed  using the Discrete Fourier Transform  in  \cite{RS1992}   (see \cite{JLS2014} and \cite{JLLX2012} for more details).  For completeness, the decomposition used in this paper is  summarized as follows.

An {\em idempotent} in $\mathcal{R}$ is a nonzero element $e$ such that $e^2=e$. It is called {\em primitive} if for every other idempotent $f$, either $ef=e$ or $ef=0$.  The primitive idempotents in $\mathcal{R}$ are induced by the $p^\nu$-cyclotomic classes of $A$ (see \cite[Proposition II.4]{DKL2000}).   Let $\{ a_1,a_2,\dots, a_t\}$ be a complete set of representatives of $p^\nu$-cyclotomic classes  of $A$  and let $e_i$ be the primitive idempotent induced by $S_{p^{\nu}}(a_i)$  for all $1\leq i\leq t$.  From \cite{RS1992},   $\mathcal{R}$ can be decomposed as 
\begin{align}\label{eq-decom0}
\mathcal{R}= \bigoplus_{i=1}^t\mathcal{R}e_i .
\end{align}
It is well known (see \cite{JLLX2012}  and \cite{JLS2014}) that $\mathcal{R}e_i\cong \mathbb{F}_{p^{\nu s_i}}$,  where  $s_i=|S_{p^{\nu}}(a_i)|$  provided that   $e_i$ is induced by $S_{p^{\nu}}(a_i)$, and hence,
\begin{align}\label{eq-decom01}
\mathbb{F}_{p^{\nu}}[A\times\mathbb{Z}_{p^k}]  
\cong \bigoplus_{i=1}^t(\mathcal{R}e_i )[\mathbb{Z}_{p^k}]\cong \prod_{i=1}^t \mathbb{F}_{p^{\nu s_i}}[\mathbb{Z}_{p^k}] \cong \prod_{i=1}^t \mathbb{F}_{p^{\nu s_i}}[x]/\langle x^{p^k}-1\rangle .
\end{align}
Therefore, every abelian code $C$ in  $  \mathbb{F}_{p^{\nu}}[A\times\mathbb{Z}_{p^k}]  $ can be viewed as   
\[ C\cong \prod_{i=1}^t C_i,\]
where $C_i$ is  a cyclic code of length $p^k$  over  $\mathbb{F}_{p^{\nu s_i}}$ for all $i=1,2,\dots, t$.

\begin{remark} \label{remE=H}
    It is well known that every cyclic code  $D$ of length $p^k$ over  $\mathbb{F}_{p^\nu}$ is uniquely  generated as ideal in $\mathbb{F}_{p^{\nu  }}[x]/\langle x^{p^k}-1\rangle$  by $(x-1)^j$ for some $0\leq j \leq p^k$. Note that such  a code  has $\mathbb{F}_{p^\nu}$-dimension $p^k-j$.  The Euclidean and Hermitian duals of $D$ are of the same form
    $D^{\perp _{\rm E}}=D^{\perp _{\rm H}}$
    generated by $(x-1)^{p^k-j}$.
\end{remark}

In order to study properties of the hulls  abelian codes in PIGAs  under  the   Euclidean and Hermitian inner products,  two rearrangements  of the ideals $\mathcal{R}e_i$  in the decomposition \eqref{eq-decom01}  will be discussed in the following subsections.

\subsection{The Average Dimension of the Euclidean Hull of Abelian Codes in PIGAs}

In this section,  we focus on an application of good integers in determining of the average dimension of the Euclidean  hulls of abelian codes in $\mathbb{F}_{p^\nu}[A\times \mathbb{Z}_{p^k}]$, where $\nu>0 $ and $k\geq 0$ are integers and $ p\nmid |A|$.

A $p^\nu$-cyclotomic class $S_{p^{\nu}}(a)$ is said to be of  {\em type} ${\I}$     if $S_{p^{\nu}}(a)=S_{p^{\nu}}(-a)$, or {\em type} ${\II}$   if $S_{p^{\nu}}(-a)\neq S_{p^{\nu}}(a)$. 	 

Without loss of generality, the representatives $a_1, a_2, \dots, a_t$ of  $p^\nu$-cyclotomic classes   of $A$  can be  chosen such that $\{a_j| j=1,2,\dots,{r_{\I}}\}$ and $\{a_{r_{\I}+l}, a_{r_{I}+r_{\II}+l}=-a_{r_{I}+l} \mid l=1,2,\dots, r_{\II}\}$ are  sets of representatives of $p^\nu$-cyclotomic classes of $A$ of types $\I$ and ${\II}$, respectively, where $t=r_{\I}+2r_{\II}$.    As assumed above,  $s_i=|S_{p^\nu}(a_i)|$ for all $1\leq i\leq t$. Clearly, $s_{r_\I+l}=s_{r_\I+r_\II+l}$ for all $1\leq l\leq r_\II$.

Rearranging the terms in the decomposition of  $\mathcal{R}$ (see \eqref{eq-decom0})  based on these  $2$ types of cyclotomic classes in \cite{JLLX2012},  we have 
\begin{align}
\mathbb{F}_{p^{\nu}}[A\times\mathbb{Z}_{p^k}]    & 
\cong   \left( \prod_{j=1}^{r_{\I}} \mathbb{K}_j[\mathbb{Z}_{p^k}]  \right) \times \left( \prod_{l=1}^{r_{\II}} (\mathbb{L}_l[\mathbb{Z}_{p^k}]  \times \mathbb{L}_l[\mathbb{Z}_{p^k}] )  \right), \label{eqSemiSim}
\end{align}
where   $ \mathbb{K}_j\cong \mathbb{F}_{p^{\nu s_{j}}}  $ for all   $j=1,2,\dots, r_{\I}$  and  $ \mathbb{L}_l  \cong \mathbb{F}_{  p^{\nu s_{r_\I+l}}}   $      for all $l=1,2,\dots, r_{\II}$.

From \eqref{eqSemiSim},  it follows that an abelian code $C$ in  $\mathbb{F}_{p^{\nu}}[A\times\mathbb{Z}_{p^k}]   $ can be viewed as 
\begin{align}\label{decomC} 
C\cong   \left(\prod_{j=1}^{r_{\I}} C_j  \right)\times \left(\prod_{l=1}^{r_{\II}} \left( D_l\times D_l^\prime\right) \right), \end{align}
where $C_j$, $D_s$ and $D_s^\prime$ are   cyclic   codes in        $\mathbb{K}_j[\mathbb{Z}_{p^k}]$, $\mathbb{L}_l[\mathbb{Z}_{p^k}]$ and $\mathbb{L}_l[\mathbb{Z}_{p^k}]$, respectively,  for all    $j=1,2,\dots,r_{\I}$ and  $l=1,2,\dots,r_{\II}$.

From   \cite[Section II.D]{JLLX2012} and Remark \ref{remE=H}, the Euclidean dual of $C$  in (\ref{decomC}) is 

\begin{align} \label{eq-Edual}
C^{\perp_{\rm E}}\cong    \left(\prod_{j=1}^{r_{\I}} C_j ^{\perp_{\rm E}} \right)\times \left(\prod_{l=1}^{r_{\II}} \left( (D_l^\prime) ^{\perp_{\rm E}}\times  D_l^{\perp_{\rm E}}\right) \right).
\end{align}

\begin{lemma} \label{correspEP}
    There is a one-to-one correspondence between  $\mathcal{C}(p^\nu,A\times\mathbb{Z}_{p^k})$ and  the set  $\{(\epsilon_1,\epsilon_2, \dots, \epsilon_t)\mid 0\leq \epsilon_i\leq p^k \text{ for all } 1\leq i\leq t\}$, where $t=r_\I+2r_\II$.
\end{lemma}
\begin{proof}
    From \eqref{decomC} and the discussion above, it is not difficult to see that the map \begin{align*}C\mapsto  (& \dim_{\mathbb{F}_{p^{\nu s_1}}}(C_1) , \dots,  \dim_{\mathbb{F}_{p^{\nu s_{r_\I}}}}(C_{r_\I}) ,\\ & \dim_{\mathbb{F}_{p^{\nu s_{r_\I+1}}}}(D_1),\dots,\dim_{\mathbb{F}_{p^{\nu s_{r_\I+r_\II}}}}(D_{r_{\II}}), \dim_{\mathbb{F}_{p^{\nu s_{r_\I+1}}}}(D_1^\prime),\dots, \dim_{\mathbb{F}_{p^{\nu s_{r_\I+r_\II}}}}(D_{r_\II}^\prime ) )\end{align*}
    is a one-to-one correspondence from  $\mathcal{C}(p^\nu,A\times\mathbb{Z}_{p^k})$ to $\{(\epsilon_1,\epsilon_2, \dots, \epsilon_t)\mid 0\leq \epsilon_i\leq p^k \text{ for all } 1\leq i\leq t\}$.
\end{proof}

Form the above discussion, Remark \ref{remE=H} and Lemma \ref{correspEP}, the next lemma follows.
\begin{lemma}\label{hullE}
    Let  $C$ be an abelian code in $\mathbb{F}_{p^\nu} [ A\times\mathbb{Z}_{p^k}]$ decomposed as in  (\ref{decomC}).  Then the following statements hold.
    \begin{enumerate}
        \item    The Euclidean hull of $C$ is    \begin{align*}
        \mathcal{H}_{\rm E}(C)\cong  \left(\prod_{j=1}^{r_{\I}}  (C_j \cap C_j ^{\perp_{\rm E}}) \right)\times \left(\prod_{l=1}^{r_{\II}} \left( (D_l\cap(D_l^\prime) ^{\perp_{\rm E}})\times  (D_l^\prime \cap D_l^{\perp_{\rm E}})\right) \right).
        \end{align*}
        \item    If  $C$ corresponds to $(\epsilon_1, \epsilon_2,\dots,\epsilon_t)$ with $t=r_\I+2r_\II$ as in Lemma \ref{correspEP}, then 
        \begin{align*} 
        \dim(\mathcal{H}_{\rm E}(C))=&\sum_{j=1} ^{r_\I} s_j\min(\epsilon_j, p^k-  \epsilon_j)   +\sum_{l=1}^{r_\II}    s_{r_\I+l}    \min(\epsilon_{r_\I+l}, p^k-  \epsilon_{r_\I+r_\II+l})     \notag
        \\
        &+\sum_{l=1}^{r_\II}    s_{r_\I+l}    \min(\epsilon_{r_\I+r_\II+l}, p^k-  \epsilon_{r_\I+l})    .
        \end{align*}
    \end{enumerate}
    
\end{lemma}

Let 
$
Q_{p^\nu}(A)=\{a\in A\mid -a\in S_{p^\nu}(a)\}.
$
It is not difficult to see that $Q_{p^\nu}(A)$  is  the union of all  $p^\nu$-cyclotomic classes of $A$ of types $\I$ and $|Q_{p^\nu}(A)|=\sum\limits_{j=1}^{r_\I} s_j$, where $s_j=|S_{p^\nu}(a_j)| $ for all $1\leq j\leq r_\I$.    

The average dimension of the Euclidean hulls of abelian codes in $\mathbb{F}_{p^\nu}[A\times \mathbb{Z}_{p^k}]$ is determined using the technique in \cite{S2003} as follows.

\begin{theorem} \label{thmAVG} Let $p$ be  a prime and let $\nu$ be a positive integer. Let $A$ be a finite abelian group of order $m$ such that $p\nmid m$. Then  
    \begin{align*}
    {\rm avg}_{p^\nu}^{\rm E}(A\times\mathbb{Z}_{p^k})= mp^k\left( \frac{1}{3}-\frac{1}{6(p^k+1)}\right)- |Q_{p^{\nu}}(A)|\left( \frac{p^k+1}{12}+\frac{2-3\delta_p}{12(p^k+1)}\right),
    \end{align*}
    where 
    $\delta_p=\begin{cases}
    1 &\text{ if }p=2,\\
    0&\text{ if }p \text{ is odd}.
    \end{cases}$
\end{theorem}
\begin{proof}
    Let $X$ be  the random variable  of the dimension  $\dim(\mathcal{H}_{\rm E}(C))$,  where $C$  is chosen randomly from $\mathcal{C}(p^\nu, A\times\mathbb{Z}_{p^k})$ with uniform probability.  Then $  {\rm avg}_{p^\nu}^{\rm E}(A\times\mathbb{Z}_{p^k})$ is the expectation $E(X)$ of $X$.
    Using Lemma \ref{hullE} and the arguments similar to those in the proof of   \cite[Proposition 22]{S2003},  $E(X)$ can be determined. Hence, the formula for $  {\rm avg}_{p^\nu}^{\rm E}(A\times\mathbb{Z}_{p^k})$ follows. 
\end{proof}

We note that if $A$ is the cyclic group of order $m$, then the average dimension of the hulls of cyclic codes of length $mp^k$  can be obtained as a special case of Theorem~\ref{thmAVG}.

Since  $0\in Q_{p^{\nu}}(A)$, we have $|Q_{p^{\nu}}(A)|\geq 1$. Hence, by Theorem \ref{thmAVG}, the next corollary follows.

\begin{corollary}\label{corBoundE1}
    Let $p$ be  a prime and let $\nu$ be a positive integer.  Let $A$ be a finite abelian group of order $m$ such that $p\nmid m$. Then  the following statements hold.
    
    \begin{enumerate}
        \item   $ {\rm avg}_{p^\nu}^{\rm E}(A\times\mathbb{Z}_{p^k})< \frac{mp^k}{3} $.

        \item   $ {\rm avg}_{p^\nu}^{\rm E}(A)=\frac{m-|Q_{p^\nu}(A)|}{4}$

        \item   $ {\rm avg}_{p^\nu}^{\rm E}(A)\leq \frac{m-1}{4}$. 
        
    \end{enumerate}
    
\end{corollary}

Let $\chi$   be a function defined by 
\[
\chi(d,p^\nu)=
\begin{cases}
1 &\text{ if } d  \in G_{(p^\nu,1)},\\
0 &\text{otherwise} .\\		
\end{cases}
\]  
The function $\chi$ and 
good  integers play an important  role in determining    ${\rm avg}_{p^\nu}^{\rm E}(A\times\mathbb{Z}_{p^k})$    in the following results.

\begin{lemma}[{\cite[Lemma 4.5]{JLLX2012} }] \label{propType} Let $p$ be  a prime and let $\nu$ be a positive integer.  Let $A$ be a finite  abelian group    such that $p\nmid |A|$  and let  $a\in A$. Then  $S_{p^\nu}(a) $ is of type $\I$  if and only if ${\rm ord}(a)\in G_{(p^\nu,1)}$. 
\end{lemma}

\begin{lemma} \label{lemQR}  Let $p$ be  a prime and let $\nu$ be a positive integer.  Let $A$ be a finite abelian group of order $m$ and exponent $M$ with  $p\nmid M$. Then   
    \begin{align*} |Q_{p^\nu}(A)|= \sum_{d\mid M}\chi(d,p^\nu)\mathcal{N}_{A}(d),
    \end{align*}
    where  $\mathcal{N}_{A}(d)$ is the number of elements of order $d$ in $A$ determined in \cite{B1997}.
    
    In particular, $|Q_{p^{\nu}}(A)|=m$ if and only if $m\in G_{(p^\nu,1)}$.
\end{lemma} 
\begin{proof} The statements follow immediately from  Lemma \ref{propType}.
\end{proof}

For each integer $\alpha\geq 0$, let  $\mathcal{P}_{p^\nu,\alpha}$ denote the set of odd primes $t$ such that $2^\alpha || {\rm ord}_{t}(p^\nu)$. For a subset $T$ of $\mathbb{N}$, denote by  $\langle\langle T \rangle\rangle$   the multiplicative semigroup generated by $T$.

For each positive integer  $m$, let $\beta\geq 0$ be the integer such that $2^\beta|| m$. Then  $m=2^\beta m^\prime$ for some odd positive integer $m^\prime$.   Since $\{\langle\langle \mathcal{P}_{p^\nu,\alpha}\rangle\rangle: \alpha \in \mathbb{N}\cup \{0\} \}$ forms a partition of the set of odd positive integers,     $m=2^\beta m^\prime$  can be   represented  as 
\begin{align*}
m=2^\beta m_0m_1m_2m_3 \dots,
\end{align*}
where $m_\alpha\in \langle\langle \mathcal{P}_{p^\nu,\alpha}\rangle\rangle$ for all $\alpha \in \mathbb{N}\cup \{0\} $ and  $m_\alpha=1$ for all but finitely many  elements $\alpha$.

\begin{proposition} \label{propFactor} Let $p$ be  a prime and let $\nu$ be a positive integer. 
    Let $A$ be a finite abelian group of order  	$m=2^\beta m_0 m_1m_2m_3 \dots$ and exponent $M=2^\mathcal{B} M_0 M_1M_2M_3 \dots $, where $m_\alpha$ and $M_\alpha$  are in   $\langle\langle \mathcal{P}_{p^\nu,\alpha}\rangle\rangle$ for all $\alpha\geq 0$. Let $\gamma\ge 0$ be  an integer such that  $2^\gamma || (p^\nu+1) $. Then  
    \begin{align}\label{QqA}
    |Q_{p^\nu}(A)|= m_1 \sum_{i=0}^\gamma \mathcal{N}_A(2^i)+(1+\mathcal{N}_A(2))^{\min\{1, \beta\}}\sum_{\alpha\geq 2}  \left(m_\alpha  - 1  \right).
    \end{align}
    (Note that $\mathcal{N}_A(2^i)$ will be regarded as $0$ if $i>\mathcal{B}$.)
\end{proposition}
\begin{proof}  
    Observe  that $M^\prime| m^\prime$ and  $M_\alpha|m_\alpha$  for all $\alpha\geq 0$ and $\beta\geq \mathcal{B}$.  
    
    \noindent  {\bf Case 1}  $\mathcal{B}\ne 0$. Then $\beta \ne 0$. By Lemma \ref{lemQR} and Corollary \ref{cor:good-gamma}, we have 
    \begin{align*}
    |Q_{p^\nu}(A)|&= \sum_{d\mid M}\chi(d,p^\nu)\mathcal{N}_{A}(d)\notag\\
    &= \sum_{d\mid M, d\in G(p^\nu,1)} \mathcal{N}_{A}(d)\notag\\
    &= \mathcal{N}_{A}(1)+\mathcal{N}_{A}(2) 
    +\sum_{d\mid 2^{\min\{\mathcal{B},\gamma\}}M_1, d\notin\{1,2\}} \mathcal{N}_{A}(d) 
    \\&~~~~   +\sum_{\alpha \geq 2}\sum_{d\mid 2M_\alpha, d\notin\{1,2\}} \mathcal{N}_{A}(d)\notag\\ 
    &= 1+\mathcal{N}_{A}(2) 
    + \left( \min\{2^{\beta}, \sum_{i=0}^\gamma \mathcal{N}_A(2^i)\}m_1-1-\mathcal{N}_{A}(2) \right)     \notag\\
    &~~~~   +\sum_{\alpha\geq 2}  ({   (1+\mathcal{N}_A(2))m_\alpha}  -1-\mathcal{N}_{A}(2) )  \notag\\
    &= 	m_1 \sum_{i=0}^\gamma \mathcal{N}_A(2^i)+\sum_{\alpha\geq 2} \left (\left(1+\mathcal{N}_{A}(2) \right)m_\alpha-1-\mathcal{N}_{A}(2) \right)   \notag\\
    &= 	 m_1 \sum_{i=0}^\gamma \mathcal{N}_A(2^i) +(1+\mathcal{N}_{A}(2))\sum_{\alpha\geq 2}  (m_\alpha  - 1  ),
    \end{align*}
    
    \noindent  {\bf Case 2}   $\mathcal{B}=0$. Then $\beta=0$.  By Lemma \ref{lemQR}, we have 
    
    \begin{align*}
    |Q_{p^\nu}(A)|&= \sum_{d\mid M}\chi(d,p^\nu)\mathcal{N}_{A}(d)\notag\\
    &= \sum_{d\mid M, d\in G(p^\nu,1)} \mathcal{N}_{A}(d)\notag\\
    &= \mathcal{N}_{A}(1)+\sum_{\alpha\geq 1}\sum_{d\mid M_\alpha,\, d> 1} \mathcal{N}_{A}(d)\notag\\
    &= \mathcal{N}_{A}(1)+\sum_{\alpha\geq 1}(m_\alpha-1)\notag\\
    &= m_1+\sum_{\alpha \geq 2}(m_\alpha-1).
    \end{align*}

    Combining the two cases, we conclude the proposition.
\end{proof}
\begin{remark}\label{remNA2}
    From Proposition \ref{propFactor}, we have the following observations. 
    \begin{enumerate}
        \item  If $A=\mathbb{Z}_{2^\beta}\times H$ is  a      group of order $m$, then 
        
        \begin{align*}\sum_{i=0}^\gamma \mathcal{N}_A(2^i) &=\begin{cases}
        2^\beta  & \text{ if } \beta\leq \gamma \\
        2^\gamma  &  \text{ if }  \gamma \leq \beta
        \end{cases}\\ &
        =2^{\min\{\gamma, \beta\}}
        \end{align*}
        which coincides with \cite[Proposition 24]{S2003}.
        \item  If $A=(\mathbb{Z}_2)^{\beta}\times H$ is  a      group of order $m$, then 
        
        \begin{align*}\sum_{i=0}^\gamma \mathcal{N}_A(2^i) &=\begin{cases}
        2^\beta & \text{ if }  \gamma\geq 1 \\
        1  &  \text{ if }  \gamma =0
        \end{cases}\\
        &
        =2^{\min\{\beta,\gamma\beta\}}.
        \end{align*}                                                                                                                                                                                                                                                          
    \end{enumerate}
\end{remark}

For a prime $q$, let $A_1 $ and $A_2$ be a finite  abelian $q$-groups of the same order $q^r$ with primary decompositions $A_1=\prod\limits_{i=1}^s\mathbb{Z}_{q^{a_i}} $ and  $A_2=\prod\limits_{i=1}^t\mathbb{Z}_{q^{b_i}} $.  The group $A_1$ is said to be {\em finer} than $ A_2$, denoted by $A_1 \preceq A_2$ , if  there exist a sequence $s_0=0< s_1<s_2<\dots<s_t=t$ and a permutation  $\rho$ on  $\{1,2,\dots,s\}$ such that  $\left(\sum\limits_{j=s_{i-1}+1}^{s_i} \rho(a_j)\right)
=b_i$ for all $i=1,2,\dots, t$.        Note that $\mathbb{Z}_{q}^r \preceq A  \preceq \mathbb{Z}_{q^r}$ for all $q$-groups $A$ of size $q^r$.

Using this concept, we have a sufficient condition to compare the values   $|Q_{p^{\nu}}(A)|$ for some finite abelian   groups $A$ of the same size  in terms of  their Sylow $2$-subgroups.
\begin{lemma}\label{corFinerE}  Let $p$ be a prime and let $\nu$ be a positive integer.  Let $A $  and $B$ be finite  abelian groups of the same order $m$ with $p\nmid m$.   If  the Sylow $2$-subgroup of    $A$   is finer than the   Sylow $2$-subgroup of  $ B$, then 
    \[ |Q_{p^{\nu}}(A)|\geq  |Q_{p^{\nu}}(B)|.\] 
    In particular,  if the   $2$-subgroup of  $A$ and   the   $2$-subgroup of  $B$  are isomorphic,  then  $ |Q_{p^{\nu}}(A)|=|Q_{p^{\nu}}(B)|$.
\end{lemma} 
\begin{proof} 
    From Proposition \ref{propFactor},   it suffices to show that  
    \begin{align}\label{finer1}
    \sum_{i=0}^\gamma \mathcal{N}_A(2^i) \geq \sum_{i=0}^\gamma \mathcal{N}_{B}(2^i) 
    \end{align}  and 
    \begin{align}\label{finer2}
    (1+\mathcal{N}_A(2))^{\min\{1, \beta\}} \geq (1+\mathcal{N}_{B}(2))^{\min\{1, \beta\}}.
    \end{align}

    Let $\mathcal{B}_A$ be the exponent of the Sylow $2$-subgroup of $A$. Then $\mathcal{B}_A$  is less than or equal to  the exponent of the Sylow $2$-subgroup of $B$. Hence,  $\mathcal{N}_A(2^i )\geq \mathcal{N}_{B}(2^i )$ for all $i=0,1,\dots, \mathcal{B}_A$. If $\gamma\leq \mathcal{B}_A$, we are done. Assume that $\mathcal{B}_A<\gamma$. Then 
    \[\sum_{i=0}^\gamma \mathcal{N}_A(2^i) = \sum_{i=0}^{\mathcal{B}_A} \mathcal{N}_A(2^i)\geq \sum_{i=0}^\gamma \mathcal{N}_B(2^i) .\]
    Therefore, the results follow.
\end{proof}

\begin{corollary}  \label{corBoundQ}  Let $p$ be  a prime and let $\nu$ be a positive integer.  Let $A$ be a finite abelian group of order  	$m=2^\beta m_0 m_1m_2m_3 \dots $, where $m_\alpha\in \langle\langle \mathcal{P}_{p^\nu,\alpha}\rangle\rangle$ for all $\alpha\geq 0$. Let $\gamma\ge 0$ be  an integer such that  $2^\gamma || (p^\nu+1) $.  Then we have 
    \begin{align*}
    2^{\min\{\beta,\gamma\}}m_1  +2^{\min\{1, \beta\}}\sum_{\alpha\geq 2}  \left(m_\alpha  - 1  \right)  & \leq |Q_{p^{\nu}}(A)|\\
    &\leq 2^{\min\{\beta,\gamma\beta\}}m_1  +2^{\min\{1, \beta\}}\sum_{\alpha\geq 2}  \left(m_\alpha  - 1  \right) .
    \end{align*}
    
\end{corollary}
\begin{proof}
    First, we note that    $(\mathbb{Z}_2)^\beta \preceq B \preceq   \mathbb{Z}_{2^\beta }$ for every abelian $2$-group $B$ of size $2^\beta$. The result is therefore follows from    Remark \ref{remNA2} and Corollary \ref{corFinerE}.
\end{proof}

From  Theorem \ref{thmAVG}, Proposition   \ref{propFactor}, and Corollary \ref{corBoundQ},  we have the following remark.

\begin{remark}\label{sumAVGE}  Let $p$ be  a prime and let $\nu$ be a positive integer. Let $A$ be a finite abelian group of order $m$ such that $p\nmid m$.    Then we have the following observations

    \begin{enumerate}
        \item     The value  $ {\rm avg}_{p^\nu}^{\rm E}(A\times\mathbb{Z}_{p^k})$ can be determined by substituting  the value of $|Q_{p^{\nu}}(A)|$ from  Proposition   \ref{propFactor} in to Theorem \ref{thmAVG}.
        
        \item Some lower and upper bounds  of  $ {\rm avg}_{p^\nu}^{\rm E}(A\times\mathbb{Z}_{p^k})$ can be computed by substituting  the bounds  of $|Q_{p^{\nu}}(A)|$ from  Corollary \ref{corBoundQ}  in to Theorem \ref{thmAVG}.
        
        \item  If the Sylow $2$-subgroup of $A$ is trivial (or equivalently, $m$ is odd), then   $   {\rm avg}_{p^{\nu}}^{\rm E}(A\times \mathbb{Z}_{p^k})$  is independent of $A$. Precisely,   $   {\rm avg}_{p^{\nu}}^{\rm E}(A\times \mathbb{Z}_{p^k})$ depends only on the cardinality  $m$ of $A$.

    \end{enumerate}

\end{remark}

Using Theorem \ref{thmAVG}, Corollary \ref{corBoundQ}, and the arguments similar to those in the proof of  \cite[Theorem 25]{S2003},  we conclude the following bounds.

\begin{corollary}\label{corBoundE}
    Let $p$ be  a prime and let $\nu$ be a positive integer. Let $A$ be a finite abelian group of order $m$ such that $p\nmid m$. Then     one of the following statements holds.
    
    \begin{enumerate}
        \item  $ {\rm avg}_{p^\nu}^{\rm E}(A\times\mathbb{Z}_{p^k})=0$ if and only if  $k=0$ and $m\in G_{(p^\nu,1)}$.
        \item  If $k>0$ or $m\notin G_{(p^\nu,1)}$, then  $ \frac{mp^k}{12} \leq   {\rm avg}_{p^\nu}^{\rm E}(A\times\mathbb{Z}_{p^k})< \frac{mp^k}{3} $.
    \end{enumerate}
\end{corollary}
%

The above results imply that  $ {\rm avg}_{p^\nu}^{\rm E}(A\times\mathbb{Z}_{p^k})$ is zero or grows as the same rate with $mp^k$. Note that if $A$  is a cyclic group, these results coincide with   \cite[Theorem~25]{S2003}.

\subsection{The Average Dimension of the Hermitian Hull of Abelian Codes in PIGAs}

As the Hermitian inner product is defined over a finite field of square order. Here we  focus on  abelian codes in  $\mathbb{F}_{{p^{2\nu}}}[A\times\mathbb{Z}_{2^k}]$, where $k\geq 0$  and $\nu\geq 1$ are  integers and $p$ is a prime. We note that  the results in this subsection can be obtained using the arguments analogous to those in Subsection 4.2. The different is the decomposition of the group algebra $\mathcal{R}:=\mathbb{F}_{{p^{2\nu}}}[A]$. Therefore, some of the proof will be omitted. For convenience, the theorem numbers are given in the form $4.N'$  if it corresponds to  $4.N$ in Subsection 4.2.

%

For each  $a\in A$, the ${p^{2\nu}}$-cyclotomic class $S_{{p^{2\nu}}}(a)$ is said to be of  {\em type} $\Ip$    if    $S_{{p^{2\nu}}}(a)=S_{p^{\nu}}(-p^\nu  a)$ or {\em type} $\IIp$  if $S_{{p^{2\nu}}}(a)\ne S_{{p^{2\nu}}}(- p^\nu a)$. 

Without loss of generality, assume that  $b_1, b_2, \dots, b_t$  are representatives of  the  ${p^{2\nu}}$-cyclotomic classes         such that $\{b_j| j=1,2,\dots,{r_\Ip}\}$  and $\{b_{r_\Ip+l}, b_{r_{\Ip}+r_{\IIp}+l}=-p^\nu b_{r_{\Ip}+l} \mid j=1,2,\dots, r_{\IIp}\}$ are  sets of representatives of $p^{2\nu}$-cyclotomic classes of $A$ of types $\Ip$  and $\IIp$, respectively, where $t=r_{\Ip}+2r_{\IIp}$.    Further, assume that $|S_{p^{2\nu}}(b_i)|=t_i$  for all $1\leq i\leq t$. Clearly, $t_{r_{\Ip}+l}=t_{r_{\Ip}+r_{\IIp}+l}$ for all $1\leq l\leq r_{\IIp}$.

Rearranging the terms in the decomposition of  $\mathcal{R}$ in  \eqref{eq-decom0} based on the above $2$ types  of $p^{2\nu}$-cyclotomic classes (see \cite{JLS2014}),   we have 
\begin{align}
\mathbb{F}_{{p^{2\nu}}}[A\times\mathbb{Z}_{p^k}]    &
\cong  \left( \prod_{j=1}^{r_{\Ip}} \mathbf{K}_j[\mathbb{Z}_{p^k}]  \right) \times \left( \prod_{l=1}^{r_{\IIp}} (\mathbf{L}_l[\mathbb{Z}_{p^k}] \times \mathbf{L}_l [\mathbb{Z}_{p^k}] )\right), \label{eqSemiSim2}
\end{align}
where  $ \mathbf{K}_j\cong\mathbb{F}_{p^{\nu2t_j}}$ for all   $j=1,2,\dots, r_{\Ip}$  and  $ \mathbf{L}_l  \cong \mathbb{F}_{p^{\nu2t_{r_\Ip+l}}} $      for all $l=1,2,\dots, r_{\IIp}$. 

{
    From \eqref{eqSemiSim2},  an abelian code $C$ in  $\mathbb{F}_{{p^{2\nu}}}[A\times\mathbb{Z}_{p^k}] $  can be viewed as 
    \begin{align}\label{decomC2} 
    C\cong   \left(\prod_{j=1}^{r_{\Ip}} \mathcal{C}_j  \right)\times \left(\prod_{l=1}^{r_{\IIp}} \left( \mathcal{D}_l\times \mathcal{D}_l^\prime\right) \right), \end{align}
    where $\mathcal{C}_j$, $\mathcal{D}_l$ and $\mathcal{D}_l^\prime$ are   cyclic   codes in    $\mathbf{K}_j[\mathbb{Z}_{p^k}]$, $\mathbf{L}_l[\mathbb{Z}_{p^k}]$ and $\mathbf{L}_l[\mathbb{Z}_{p^k}]$, respectively,  for all   $j=1,2,\dots,r_{\Ip}$ and  $l=1,2,\dots,r_{\IIp}$.

    In  \cite[Section II.D]{JLS2014} and Remark \ref{remE=H}, the Hermitian  dual of $C$  in (\ref{decomC2}) is 

    \begin{align} 
    C^{\perp_{\rm H}}\cong   \left(\prod_{j=1}^{r_{\Ip}} \mathcal{C}_j ^{\perp_{\rm E}} \right)\times \left(\prod_{l=1}^{r_{\IIp}} \left( (\mathcal{D}_l^\prime) ^{\perp_{\rm E}}\times  \mathcal{D}_l^{\perp_{\rm E}}\right) \right).
    \end{align}

}

\renewcommand{\thelemma}{\ref{correspEP}$'$}
\begin{lemma}  \label{correspHP}
    There is a one-to-one correspondence between  $\mathcal{C}(p^{2\nu},A\times\mathbb{Z}_{p^k})$ and  $\{(\epsilon_1,\epsilon_2, \dots, \epsilon_t)\mid 0\leq \epsilon_i\leq p^k \text{ for all } 1\leq i\leq t\}$, where $t=r_\Ip+2r_\IIp$.
\end{lemma}

The following corollary can be obtained using Remark \ref{remE=H}, Lemma \ref{correspHP} and the above discussion.

\renewcommand{\thelemma}{\ref{hullE}$'$}
\begin{lemma} \label{hullH}
    Let  $C$ be an abelian code in $\mathbb{F}_{p^{2\nu}} [ A\times\mathbb{Z}_{p^k}]$ decomposed as (\ref{decomC2}).  Then  the following statements hold.
    \begin{enumerate}
        \item    The Hermitian hull of $C$ 
        \begin{align}
        \mathcal{H}_{\rm H}(C)&\cong\left(\prod_{j=1}^{r_\Ip} (\mathcal{C}_j \cap  \mathcal{C}_j^{\perp_{\rm E}})  \right) \times \left(\prod_{s=1}^{r_{\IIp}} \left( (\mathcal{D}_s\cap(\mathcal{D}_s^\prime) ^{\perp_{\rm E}})\times  (\mathcal{D}_s^\prime \cap \mathcal{D}_s^{\perp_{\rm E}})\right) \right).
        \end{align}
        
        \item    If  $C$ corresponds to $(\epsilon_1, \epsilon_2,\dots,\epsilon_t)$ with $t=r_\Ip+2r_\IIp$ as in Lemma \ref{correspHP}, then 
        \begin{align*} 
        \dim(\mathcal{H}_{\rm H}(C))=&\sum_{j=1} ^{r_\Ip} t_j\min(\epsilon_j, p^k-  \epsilon_j) \notag  +\sum_{l=1}^{r_\IIp}    t_{r_\Ip+l}    \min(\epsilon_{r_\Ip+l}, p^k-  \epsilon_{r_\Ip+r_\IIp+l})     \notag
        \\
        &+\sum_{l=1}^{r_\IIp}    t_{r_\Ip+l}    \min(\epsilon_{r_\Ip+r_\IIp+l}, p^k-  \epsilon_{r_\Ip+l})    .
        \end{align*}
    \end{enumerate}
    
\end{lemma}

Let  
$
R_{{p^{2\nu}}}(A)=\{a\in A\mid -p^\nu \cdot a\in S_{{p^{2\nu}}}(a)\}.
$
It is not difficult to see   that   $R_{{p^{2\nu}}}(A)$ is  the union of all   ${p^{2\nu}}$-cyclotomic classes of $A$ of type $\Ip$ and $|R_{p^{2\nu}}(A)|=\sum\limits_{j=1}^{r_\Ip} t_j$, where $t_j=S_{p^{2\nu}}(b_j)$ for all $1\leq j\leq r_\Ip$.

\renewcommand{\thetheorem}{\ref{thmAVG}$'$}
\begin{theorem} \label{thmAVGH} Let $p$ be  a prime and let $\nu$ be a positive integer. Let $A$ be a finite abelian group of order $m$ such that $p\nmid m$. Then   
    \begin{align*}
    {\rm avg}_{{p^{2\nu}}}^{\rm H}(A\times\mathbb{Z}_{p^k})=mp^k\left( \frac{1}{3}-\frac{1}{6(p^k+1)}\right)- |R_{{p^{2\nu}}}(A)|\left( \frac{p^k+1}{12}+\frac{2-3\delta_p}{12(p^k+1)}\right),
    \end{align*}
    where 
    $\delta_p=\begin{cases}
    1 &\text{ if }p=2,\\
    0&\text{ if }p \text{ is odd}.
    \end{cases}$
\end{theorem}

\renewcommand{\thecorollary}{\ref{corBoundE1}$'$}
\begin{corollary}\label{corBoundH1}
    Let $p$ be  a prime and let $\nu$ be a positive integer.  Let $A$ be a finite abelian group of order $m$ such that $p\nmid m$. Then  the following statements hold.
    
    \begin{enumerate}

        \item  $ {\rm avg}_{{p^{2\nu}}}^{\rm H}(A\times\mathbb{Z}_{p^k})< \frac{mp^k}{3} $. 
        
        \item  $ {\rm avg}_{{p^{2\nu}}}^{\rm H}(A)= \frac{m-|R_{{p^{2\nu}}}(A)|}{4}$. 
        \item  $ {\rm avg}_{{p^{2\nu}}}^{\rm H}(A)\leq \frac{m-1}{4}$. 
    \end{enumerate}
    
\end{corollary}

Let   $\lambda$ be the  function defined by

\begin{align*}
\lambda(d,q)=
\begin{cases}
1 &\text{ if } d \in OG_{(q,1)},\\
0 &\text{otherwise.} \\		
\end{cases}
\end{align*}
Then the function $\lambda$ and oddly-good integers play a role in determining         $ {\rm avg}_{{p^{2\nu}}}^{\rm H}(A\times\mathbb{Z}_{p^k})$ in the following results.

\renewcommand{\thelemma}{\ref{propType}$'$}
\begin{lemma}[{\cite[Lemma 3.5]{JLS2014} }] \label{propTypeH}  Let $p$ be  a prime and let $\nu$ be a positive integer. Let $A$ be a finite  abelian group    with $ p\nmid |A|$  and let  $a\in A$. Then  
    $S_{{p^{2\nu}}}(a) $ is of type $\Ip$  if and only if ${\rm ord}(a)\in OG_{(p^\nu,1)}$.
\end{lemma}

From Lemma \ref{propTypeH}, we have the following result.

\renewcommand{\thelemma}{\ref{lemQR}$'$}
\begin{lemma} \label{lemQRH}  Let $p$ be  a prime and let $\nu$ be a positive integer. Let $A$ be a finite abelian group of order $m$ and exponent $M$ with $p\nmid m$. Then   
    \begin{align*} |R_{{p^{2\nu}}}(A)|= \sum_{d\mid M}\lambda(d,p^\nu)\mathcal{N}_{A}(d),
    \end{align*}
    where  $\mathcal{N}_{A}(d)$ is the number of elements of order $d$ in $A$ determined in \cite{B1997}.
\end{lemma}

\renewcommand{\theproposition}{\ref{propFactor}$'$}
\begin{proposition} \label{propFactorH} Let $p$ be  a prime and let $\nu$ be a positive integer. 
    Let $A$ be a finite abelian group of order  	$m=2^\beta m_0 m_1m_2m_3 \dots$, where $m_\alpha\in \langle\langle \mathcal{P}_{p^\nu,\alpha}\rangle\rangle$ for all $\alpha\geq 0$. Let $\gamma\ge 0$ be  an integer such that  $2^\gamma || (p^\nu+1) $. Then  
    \begin{align}\label{Rq2A}
    |R_{{p^{2\nu}}}(A)|=  m_1\sum_{i=0}^\gamma \mathcal{N}_A(2^i) . \notag
    \end{align}
\end{proposition}
\begin{proof} Let  $	M=2^\mathcal{B} M_0 M_1M_2M_3 \dots $ be the exponent of $A$,   where $M_\alpha$  is in   $\langle\langle \mathcal{P}_{p^\nu,\alpha}\rangle\rangle$ for all $\alpha\geq 0$.
    Observe  that $M^\prime| m^\prime$ and  $M_\alpha|m_\alpha$  for all $\alpha\geq 0$ and $\beta\geq \mathcal{B}$.   
    It follows that 
    
    \begin{align}
    |R_{{p^{2\nu}}}(A)|&= \sum_{d\mid M}\lambda(d,q)\mathcal{N}_{A}(d)\notag\\
    &= \sum_{d\mid M, d\in OG(q,1)} \mathcal{N}_{A}(d)\notag\\
    &= \sum_{d\mid  2^{\min\{\mathcal{B},\gamma\}} M_1} \mathcal{N}_{A}(d) \notag\\
    &=  m_1\sum_{i=0}^\gamma \mathcal{N}_A(2^i). \notag
    \end{align}
    as desired.
\end{proof}

\renewcommand{\thelemma}{\ref{corFinerE}$'$}

\begin{lemma} \label{corFinerH} Let $p$ be  a prime and let $\nu$ be a positive integer. Let $A $ and $\mathcal{A}$ be an abelian groups of the same order $m$ and $p\nmid m$.   If   the Sylow $2$-subgroup of  $A$   is finer than  the Sylow $2$-subgroup of $ B$, then  
    
    \[ |R_{{p^{2\nu}}}(A)|\geq|R_{{p^{2\nu}}}(B)| .\]
\end{lemma}

From Lemma \ref{corFinerH} and Remark \ref{remNA2},  we have the following bounds for $ |R_{{p^{2\nu}}}(A)|$.

\renewcommand{\thecorollary}{\ref{corBoundQ}$'$}

\begin{corollary} \label{corBoundQH}   Let $p$ be  a prime and let $\nu$ be a positive integer.  Let $A$ be a finite abelian group of order  	$m=2^\beta m_0 m_1m_2m_3 \dots $, where $m_\alpha$ is in   $\langle\langle \mathcal{P}_{p^\nu,\alpha}\rangle\rangle$ for all $\alpha\geq 0$. Let $\gamma\ge 0$ be  an integer such that  $2^\gamma || (p^\nu+1) $.  Then we have  
    
    \[2^{\min\{\beta,\gamma\}} m_1  \leq |R_{{p^{2\nu}}}(A)|\leq  2^{\min \{\beta, \gamma\beta\}} m_1.\]
    
\end{corollary}

\renewcommand{\theremark}{\ref{sumAVGE}$'$}
\begin{remark}
    Let $p$ be  a prime and let $\nu$ be a positive integer. Let $A$ be a finite abelian group of order $m$ such that $p\nmid m$.    Then we have the following observations

    \begin{enumerate}
        \item     The value  $ {\rm avg}_{p^\nu}^{\rm E}(A\times\mathbb{Z}_{p^k})$ can be determined by substituting  the value of $|Q_{p^{\nu}}(A)|$ from  Proposition   \ref{propFactorH} in to Theorem \ref{thmAVGH}.
        
        \item Some lower and upper bounds  of  $ {\rm avg}_{p^{2\nu}}^{\rm H}(A\times\mathbb{Z}_{p^k})$ can be computed by substituting  the bounds  of $|R{p^{2\nu}}(A)|$ from  Corollary \ref{corBoundQH}  in to Theorem \ref{thmAVGH}.
        
        \item  If the Sylow $2$-subgroup of $A$ is trivial (or equivalently, $m$ is odd), then   $   {\rm avg}_{p^{2\nu}}^{\rm H}(A\times \mathbb{Z}_{p^k})$  is independent of $A$. Precisely,   $   {\rm avg}_{p^{2\nu}}^{\rm H}(A\times \mathbb{Z}_{p^k})$ depends only on the cardinality  $m$ of $A$.

    \end{enumerate}
    
\end{remark}

Using Theorem \ref{thmAVGH}, Corollary \ref{corBoundQH}, and the arguments similar to those in the proof of  \cite[Theorem 25]{S2003},  we deduce the following bounds.

\renewcommand{\thecorollary}{\ref{corBoundE}$'$}
\begin{corollary}
    Let $p$ be  a prime and let $\nu$ be a positive integer. Let $A$ be a finite abelian group of order $m$ such that $p\nmid m$. Then one of  the following statements hold.
    
    \begin{enumerate}
        \item  $ {\rm avg}_{{p^{2\nu}}}^{\rm H}(A\times\mathbb{Z}_{p^k})=0$ if and only if  $k=0$ and  $m\in OG_{(q,1)}$.

        \item If $k>0$ or $m\notin OG_{(p^\nu,1)}   $,  then      $ \frac{mp^k}{8} \leq   {\rm avg}_{p^{2\nu}}^{\rm H}(A\times\mathbb{Z}_{p^k})< \frac{mp^k}{3} $.  
    \end{enumerate}
    
\end{corollary}

The above results imply   that  $ {\rm avg}_{p^{2\nu}}^{\rm H}(A\times\mathbb{Z}_{p^k})$ is zero or grows as the same rate with $mp^k$. Note that if $A$  is a cyclic group, these results coincide with   \cite[Theorem~4.9]{JS2016}.

\section{Conclusion}
A class of good integers  introduced  in $1997$ has been reconsidered.   A complete characterization of  arbitrary  good integers have been  given.  Two subclasses of   good integers have been  introduced, namely, oddly-good and evenly-good integers.  Characterization and properties of  good   integers in these two subclasses  have been determined as well.

As applications,  good integers have been linked with structures and problems  in coding theory (see,  for example,  \cite{JLX2011}, \cite{JLLX2012},  and  \cite{KG1969}).  Here,  the hulls of abelian codes have been  studied using good and oddly-good integers. The   average dimension of the hulls of abelian codes has been  determined under both the Euclidean and Hermitian inner products.     The results for cyclic codes in \cite{S2003} can be viewed as corollaries.



\begin{thebibliography}{99}
    
    \bibitem{B1997}  Benson, S.: Students ask the darnedest things: A result in elementary group theory. 
    {Math. Mag.}  {\bf 70}, 207--211  (1997).
    
    \bibitem{Be1967_2}  Berman,  S. D.:
    Semi-simple cyclic and abelian codes.
    {Kibernetika}  {\bf 3}, 21--30 (1967).
    
    
    
    \bibitem{DMS2007}  {Dicuangco, L., Moree, P., Sol\'e, P.:} {The lengths of Hermitian self-dual extended duadic codes.}  { J. Pure Appl. Algebra}  {\bf 209},  223--237 (2007).
    
    
    \bibitem{DMS2014}   {Dicuangco-Valdez, L., Moree, P., Sol\'e, P.:} {On the existence of Hermitian self-dual extended abelian group codes.} {Springer Proceedings in Mathematics and Statistics} {\bf 115},  67--84 (2014) .
    
    
    \bibitem{DKL2000}  Ding, C.,  Kohel, D. R.,   Ling, S.:
    Split group codes. 
    {IEEE Trans. Inform. Theory}   {\bf 46}, 485--495 (2000).
    
    
    
    
    \bibitem{FiSe1976}  Fisher, J. L.,  Sehgal, S. K.:
    Principal ideal group rings.
    {Comm. Algebra}  {\bf 4},  319--325 (1976).
    
    
    
    \bibitem{JLX2011}  Jia, Y.,  Ling, S.,  Xing, C.:  On self-dual cyclic codes over finite fields.
    {IEEE Trans. Inform. Theory}   {\bf  57},  2243--2251 (2011).
    
    
    \bibitem{J2017}   Jitman, S.:  Good integers and some applications in coding theory,    Cryptogr. Commun. \textbf{10},    685--704 (2018).
    
    \bibitem{JLLX2012}  Jitman, S.,  Ling, S.,  Liu, H.,  Xie, X.:  Abelian codes in principal ideal group algebras. {IEEE Trans. Inform. Theory}  {\bf  59}, 3046--3058 (2013). 
    
    
    
    \bibitem{JLS2014}  Jitman, S.,  Ling, S., Sol\'e, P.: Hermitian self-dual Abelian codes.  {IEEE Trans. Inform. Theory}  {\bf  60}, 1496--1507 (2014).
    
    
    \bibitem{JS2016}    Jitman, S., Sangwisut, E.:  The average dimension of the Hermitian hull of cyclic codes over finite fields of square order.  {AIP Conference Proceedings}  {\bf 1775}, 030026 (2016).
    
    
    
    
    
    
    \bibitem{KG1969}  Knee, D., Goldman, H. D.: Quasi-self-reciprocal polynomials and potentially large minimum distance BCH codes.  {IEEE Trans. Inform. Theory}  {\bf 15}, 
    118 --121 (1969).
    
    
    \bibitem{Leon82} Leon,  J. S.: Computing automorphism groups of error-correcting codes. {IEEE Trans. Inform. Theory}  {\bf 28},  496--511 (1982).
    
    \bibitem{Leon91} Leon, J. S.: Permutation group algorithms based on partition  I: theory and algorithms.   {J. Symbolic Comput.}  {\bf 12},  533--583 (1991).
    
    \bibitem{Leon97}  Leon, J. S.: Partitions, refinements, and permutation group computation.  {Discrete Math. Theoret. Comput. Sci.}  {\bf 28},  123--158 (1997).
    
    \bibitem{M1997}  Moree, P.: On the divisors of $a^k+b^k$. {Acta Arithmetica} {\bf LXXX}, 197--212 (1997).
    
    
    
    
    \bibitem{RS1992}  Rajan, B. S., Siddiqi,  M. U.:
    Transform domain characterization of abelian codes.
    {IEEE Trans. Inform. Theory}   {\bf  38}, 1817--1821  (1992).
    
    
    \bibitem{SJLP2015}   Sangwisut, E.  Jitman, S.,  Ling, S.,   Udomkavanich, P.: Hulls of cyclic and negacyclic codes over finite fields.  {Finite Fields Appl.}   {\bf 
        33}, 232--257  (2015).
    
    
    
    \bibitem{Sendrier997}  Sendrier, N.: Finding the permutation between equivalent binary code. 
    {Proc. IEEE ISIT'1997}, Ulm, Germany,   367 (1997).
    
    
    
    \bibitem{Sendrier97}  Sendrier,  N.: On the dimension of the hull. {SIAM J. Appl. Math.}  {\bf 10},  282--293 (1997).
    
    \bibitem{Sendrier00}  Sendrier, N.: Finding the permutation between equivalent codes: the support splitting algorithm. {IEEE Trans. Inform. Theory}  {\bf 46}, 1193--1203 (2000).
    
    \bibitem{Sendrier01} Sendrier, N.,  Skersys, G.: On the computation of the automorphism group of a linear code. {Proc. IEEE ISIT'2001}, Washington, DC,  13 (2001).
    
    
    \bibitem{S2003}  Skersys, G.: The average dimension of the hull of cyclic codes.  {Discrete Appl. Math.}  {\bf  128}, 275--292 (2003).
    
\end{thebibliography}
\end{document}